\documentclass{article}
\usepackage{graphicx}
\usepackage{camnum}
\usepackage{amsmath}
\usepackage{harvard}
\usepackage{xcolor}
\usepackage{algorithm}
\usepackage[noend]{algpseudocode}

\newtheorem{assumption}[theorem]{Assumption}

\allowdisplaybreaks

 \newcommand{\black}{\color{black}}

\newcommand{\ee}{{\mathrm e}}
\newcommand{\ii}{{\mathrm i}}

\begin{document}
\title{Numerical solution of the linear semiclassical Schr\"odinger equation on the real line} 
\author{Arieh Iserles\\
Department of Applied Mathematics and Theoretical Physics\\
Centre for Mathematical Sciences\\
University of Cambridge\\
Wilberforce Rd, Cambridge CB4 1LE\\
United Kingdom
 \and
Karolina Kropielnicka\\
Institute of Mathematics\\
Polish Academy of Sciences\\
Antoniego Abrahama 18, 81-825 Sopot\\
Poland
 \and
Katharina Schratz\\
Laboratoire Jacques-Louis Lions\\
Sorbonne Universit\'e\\
4 place Jussieu, 75252 Paris\\
France
 \and 
Marcus Webb\\
Department of Mathematics\\
University of Manchester\\
Alan Turing Building\\
Manchester M13 9PL \\
United Kingdom}
 
\thispagestyle{empty}
\maketitle

\begin{abstract}
  The numerical solution of a linear Schr\"odinger equation in the semiclassical regime is very well understood in a torus $\BB{T}^d$. A raft of modern computational methods are precise and affordable, while conserving energy and resolving high oscillations very well. This, however, is far from the case with regard to its solution in $\BB{R}^d$, a setting more suitable for many applications. In this paper we extend the theory of splitting methods to this end. The main idea is to derive the solution using a spectral method from a combination of solutions of the free Schr\"odinger equation and of linear scalar ordinary differential equations, in a symmetric Zassenhaus splitting method. This necessitates detailed analysis of certain orthonormal spectral bases on the real line and their evolution under the free Schr\"odinger operator. 
\end{abstract}

\section{Introduction}
\subsection{Why the real line?}

This paper is concerned with the numerical solution of the linear Schr\"odinger equation in the semiclassical regime, describing the motion of an electron in a quantum system,
\begin{equation}
  \label{eq:1.1}
  \ii \varepsilon \frac{\partial u}{\partial t}=-\varepsilon^2 \Delta u+V(\MM{x})u,\qquad t\geq0,\; \MM{x}\in\BB{R}^d,
\end{equation}
where the initial condition $u(\MM{x},0)=u_0(\MM{x})\in\CC{L}_2(\BB{R}^d)$ for all $\MM{x}\in\BB{R}^d$ is given. The semiclassical parameter $\varepsilon>0$ is a small number which describes the square root of the ratio between the mass of an electron and the total mass of the system, and $V:\BB{R}^d\rightarrow\BB{R}$ is the interaction potential which is assumed to be smooth for the purposes of this paper. Since $|u(\MM{x},t)|^2$ gives the probability density of the electron residing at $\MM{x}$ at time $t$, the system is required to satisfy,
\begin{equation}
  \label{eq:1.2}
  \int_{\bb{R}^d} |u(\MM{x},t)|^2\D\MM{x}\equiv1,
\end{equation}
and any physically relevant numerical solution must be consistent with this conservation law. To read more, \cite{lasser20cqd} is an excellent, up to date review of both the equation \R{eq:1.1} and its numerical solution. 

Respecting the unitarity property \R{eq:1.2} underlies the importance of geometric numerical integration methodologies in this context and has been central to modern treatment of the linear Schr\"odinger equation in  the semiclassical,  $0<\varepsilon\ll1$, and the atomistic,  $\varepsilon=1$, regimes alike \cite{bader14eas,blanes17hoc,iserles18mlm,iserles19sse,jin11mcm}. However, all these publications are focussed on a subtly different problem: instead of being defined in $\BB{R}^d$, the equation \R{eq:1.1} is set on a torus, typically $\BB{T}^d$, with periodic boundary conditions. This is of crucial importance to \emph{splitting techniques,\/} a common denominator to all these methodologies, because the {\em free Schr\"odinger equation}
\begin{equation}
  \label{eq:1.3}
  \ii \frac{\partial u}{\partial t}=-\varepsilon \Delta u,
\end{equation}
given with periodic boundary conditions, can be approximated very rapidly, affordably and precisely by means of the Fast Fourier Transform (FFT).

Our contention is that the periodic setting imposes unwelcome limitations on the solution, which might lead to altogether false outcomes, and this becomes problematic once a solution over a long time interval  is sought (e.g.\ in quantum control). The underlying reason is the tension arising from the nature of the differential equation and of the initial condition, both predicated by quantum-mechanical considerations. The differential equation itself is dispersive: different waves travel at different speeds, dependent on their wavelengths, which can span a very wide range, all the way from $\O{1}$ to $\O{\varepsilon^{-1}}$.  The initial condition is typically a linear combination of highly localised (and rapidly oscillating) {\em wave packets.\/} Recall that $|u(\MM{x},t)|^2$ represents the probability of a particle residing at $\MM{x}$ in time $t$: while it is a central tenet of quantum mechanics that a particle cannot be completely localised, typically $|u(\MM{x},t)|^2$ is a linear combination of narrowly-concentrated Gaussian-like structures. These Gaussian-like structures travel at different speeds and, provided the equation is solved for sufficiently long time, some of them eventually reach the boundary. At this very moment periodicity becomes a foe because the wave packet reaches the boundary and `pops out' at the other end --- this is not physical! 

An alternative to periodic boundary conditions is to impose zero Dirichlet or zero Neumann boundary conditions. However, the following argument shows that this approach is also problematic. Consider an initial condition $u_0 \in \CC{H}^1_0(0,1)$ and potential $V \in \CC{H}^1(0,1)$. Now consider the following two initial boundary value problems, the first of which has zero Dirichlet boundary conditions, the second of which has periodic boundary conditions:
\begin{Eqnarray}
  \ii \varepsilon \frac{\partial v}{\partial t} &=& -\varepsilon^2  \frac{\partial^2 v}{\partial x^2} + V(x) v, \qquad x \in [0,1], \label{eq:DBC1} \\
  \nonumber
  v(0,t) &=& 0, \qquad v(1,t) = 0, \qquad t > 0, \\
  \nonumber
  v(x,0) &=& u_0(x), \qquad x \in [0,1],
  \end{Eqnarray}
and,
\begin{Eqnarray}
  \ii \varepsilon \frac{\partial w}{\partial t} &=& -\varepsilon^2  \frac{\partial^2 w}{\partial x^2} + V(|x|) w, \qquad x \in [-1,1], \label{eq:periodised1} \\
  \nonumber
  w(-1,t) &=& w(1,t), \qquad \partial_x w(-1,t) = \partial_x w(1,t), \qquad t > 0, \\
  \nonumber
  w(x,0) &=& \mathrm{sign}(x) u_0(|x|), \qquad x \in [-1,1] .
  \end{Eqnarray} 
 The relationship between $v(x,t)$ and $w(x,t)$ for $x \in [0,1]$ and $t > 0$ is rather simple. Clearly the oddness of $w(x,0)$ is preserved since the second derivative and multiplication by $V(|x|)$ preserve oddness. Combining oddness with periodicity implies that $w(0,t) = 0 = w(1,t)$ for all time. It therefore follows from uniqueness of solution to \R{eq:DBC1} that $w(x,t) = v(x,t)$ for $x \in [0,1]$ and $t > 0$. So now let us return to the notion of a wave packet moving towards the boundary, but this time with zero Dirichlet boundary conditions imposed. The solution to the odd extension implies that this wave packet will be reflected back and its sign reversed --- while this physically happens in the case of an infinite potential barrier, it is not the correct behaviour when posed in free space! A similar construction can be made for Neumann boundary conditions.

\begin{figure}[h!]
  \centering
  \includegraphics[width=.95\textwidth]{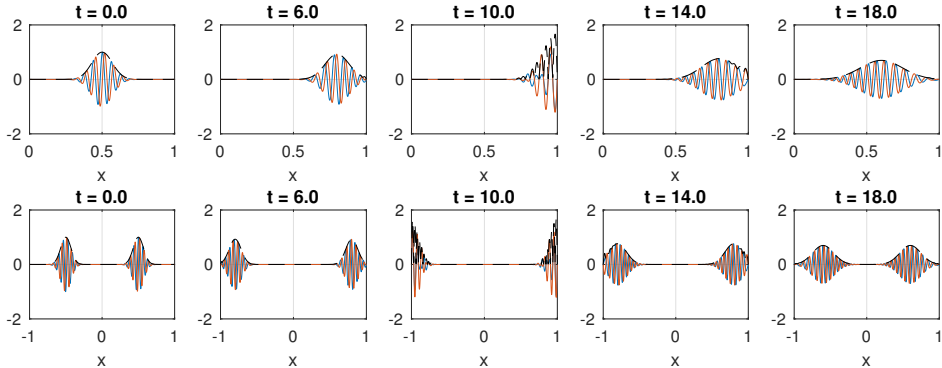}
  \caption{Top: We plot the evolution of \eqref{eq:DBC1} with $u_0(x)$ an approximate wave packet (so that zero boundary conditions are satisfied) and $V(x) = 0$. The wave packet moves rightwards towards the right boundary until time $t = 10$, after which it moves leftwards, returning unscathed by the encounter. Such ``reflections'' contradict the behaviour of a wave packet in free space. Bottom: We plot the evolution of the corresponding extension in \eqref{eq:periodised1}. We see that the reflection behaviour for the Dirichlet initial boundary value problem can be explained by the periodic behaviour of this one.}
  \end{figure}

We hope this has convinced the reader: no matter what we do, and no matter how rapidly and accurately we can solve Schr\"odinger's equation posed on a bounded set, the result of truncating the domain from $\BB{R}^d$ to such a set destroys the physics of the problem over a large enough time interval. This is the {\em raison d'\^{e}tre\/} for this paper: solve \R{eq:1.1} without compromising its setting in $\BB{R}^d$. Throughout the remainder of the paper we assume that \R{eq:1.1} is presented in a single space dimension, $d=1$. A generalisation to a modest number of space dimensions can be accomplished with tensor products along the lines of \cite{bader14eas}, while generalisation to a large number of dimensions would require a raft of additional techniques and is beyond the scope of the current paper.

 To achieve this aim, we will extend the framework of symmetric Zassenhaus splittings, which has been developed for \eqref{eq:1.1} on the torus $\BB{T}$ \cite{bader14eas,singh}, to \eqref{eq:1.1} posed on the whole real line. This is not a straightforward exercise, because we cannot use special properties of the Fourier basis. In Section \ref{sec:Zassenhaus} we derive these Zassenhaus splittings under more general assumptions, allowing for bases other than the Fourier basis to be used. In Section \ref{sec:orthoFSE}, we discuss the solution of the free Schr\"odinger equation \eqref{eq:1.3}, focusing on two bases which are orthonormal on the real line: Hermite functions and Malmquist--Takenaka functions. In Section \ref{sec:conclusion} we demonstrate how these pieces can be put together to construct practical numerical solvers on the real line.

\section{Splitting techniques}\label{sec:Zassenhaus}

For the clarity of exposition we  write $\partial_x^2$ instead of $\Delta$ as in \eqref{eq:1.1}. The simplest splitting methodology is to separate the potential and kinetic parts in \R{eq:1.1}, $\ii\varepsilon \partial_x^2 u -\ii\varepsilon^{-1}V(x)u$, building upon the fact that separate solutions of 
\begin{displaymath}
   \frac{\partial u}{\partial t}=-\ii\varepsilon^{-1}V(x)u\qquad \mbox{and}\qquad \frac{\partial u}{\partial t}=\ii\varepsilon \partial_x^2 u
\end{displaymath}
are (at least in a torus or a parallelepiped) much less expensive to compute than those of the full problem. We abuse notation for the exponential and write
\begin{displaymath}
  u(x,t)=\ee^{-\ii t\varepsilon^{-1}V(x)}u(x,0) \qquad \mbox{and}\qquad u(x,t)=\ee^{\ii t \varepsilon \partial_x^2} u(x,0)
\end{displaymath}
for their respective solutions. Splitting methods produce a sequence of functions $u^{0}(x)$, $u^{1}(x)$, $u^2(x)$,$\ldots$, intended to satisfy $u^k(x) \approx u(x,kh)$ where $h$ is the time-step parameter. These functions of $x$ can be discretised by any approach, for example by a spectral method.

The two simplest splitting methods are the Lie--Trotter formula
\begin{equation}
  u^{k+1}(x) = \ee^{\ii\varepsilon h \partial_x^2} \ee^{-\ii h \varepsilon^{-1}V(x)} u^{k}(x),
  \end{equation}
and
\begin{equation}
  \label{eq:1.4}
  u^{k+1}(x)=\ee^{-\ii h \varepsilon^{-1}V(x)/2} \ee^{\ii\varepsilon h \partial_x^2} \ee^{-\ii h \varepsilon^{-1}V(x)/2} u^k(x).
\end{equation}
Of course, the role of $\varepsilon^{-1}V(x)$ and $\varepsilon \partial_x^2$ can be reversed. The latter approach, advocated in \cite{jin11mcm} in tandem with spectral methods, is the famous {\em Strang splitting\/} (known also as Strang--Marchuk splitting in Russian literature).

Formally, the Strang splitting is known to produce time-stepping methods bearing an error of $\O{h^3}$. However, this is misleading because the error constant  is of size $\O{\varepsilon^{-1}}$, as we  show below using Theorem \ref{thm1}. A more effective measure of error should incorporate the small parameter $\varepsilon$, which may be even smaller in magnitude than the time-step $h$. To calculate the effective error of the splitting~(\ref{eq:1.4}), where the error constant does not depend on the small semiclassical parameter $\varepsilon$, let us have a closer look at symmetric Baker--Campbell--Hausdorff formula \cite[Sec.~III.4.2]{hairer2006geometric},
\begin{equation}\label{sBCH}
  \ee^{\frac12 \tau A}\ee^{\tau B}\ee^{\frac12 \tau A}=\ee^{\mathrm{sBCH}(\tau A,\tau B)}
\end{equation}
  where $A=-\varepsilon^{-1}V$, $B=\varepsilon \partial_x^2$ and $\tau=\ii h$ with 
  \begin{align*}
    \mathrm{sBCH}(\tau A,\tau &B)=\tau A+\tau B-\tau^{3}\frac{1}{24}[[B,A],A]-\tau^{3}\frac{1}{12}[[B,A],B]\\
    &+\tau^{5}\frac{7}{5760}[[[[B,A],A],A],A]+\tau^{5}\frac{7}{1440}[[[[B,A],A],A],B]\\
    &+\tau^{5}\frac{1}{180}[[[[B,A],A],B],B]+\tau^{5}\frac{1}{720}[[[[B,A],B],B],B]\\
    &+\tau^{5}\frac{1}{480}[[[B,A],A],[B,A]]-\tau^{5}\frac{1}{360}[[[B,A],B],[B,A]]+{\rm h.o.t.}
  \end{align*}  
Given that $A$ and $B$ are unbounded operators and also contain powers of $\varepsilon$, we now proceed to clarify the meaning of ``h.o.t.'' (higher order terms).
  \subsection{A new analysis of the sBCH formula for the semiclassical Schr\"odinger equation}
  As it was shown in \cite{jin11mcm} Schr\"odinger equations in semiclassical regime produce oscillations in space of frequency $\O{\varepsilon^{-1}}$, which places restrictions on the discretisation in space depending on which basis is used, because we must employ  sufficiently fine discretisation to resolve these oscillations. If the spatial variable is discretised using the Fourier basis then this necessitates $\mathcal{O}(\varepsilon^{-1})$ basis elements, which in turn, leads to the conclusion that after discretisation, operators of type $\partial_x^n$ have a spectral radius which scales like $\O{\varepsilon^{-n}}$. As we  discuss in Section \ref{sec:orthoFSE}, for other bases it is not necessarily the case that $\mathcal{O}(\varepsilon^{-1})$ basis elements can resolve spatial oscillations of frequency $\O{\varepsilon^{-1}}$ (indeed the Fourier basis is the optimal basis for resolving periodic oscillations). As such, we will not make assumptions about the number of basis elements, but rather, make assumptions directly on the spectral radius of the partial derivative operator (an assumption which holds in both of our examples discussed in Section \ref{sec:orthoFSE})
  \begin{assumption}\label{ass:spectral}
    Throughout this paper we will assume that after spatial discretisation, the operator $\partial_x$ has spectral radius $\mathcal{O}(\varepsilon^{-1})$.
    \end{assumption} 
  
  Since the potential $V(x)$ can in principle be an unbounded function on the real line, we must be careful that our expansions be treated \emph{locally} in $x$.
  \begin{assumption}\label{ass:potential}
    The potential $V:\BB{R}\to\BB{R}$ is infinitely differentiable, which we write $V \in \CC{C}^\infty_{\mathrm{loc}}(\BB{R})$. As a result, all derivatives are locally bounded in $\BB{R}$.
    \end{assumption}
  
  We can now make sense of ``h.o.t.'' in the sBCH formula by bounding the magnitude element of the \emph{Hall basis} for the free Lie algebra generated by $A=-\varepsilon^{-1}V$ and  $B=\varepsilon \partial_x^2$ (i.e.~$A$, $B$, $[A,B]$, $[[B,A],A]$, $[[B,A],B]$ $\ldots$ \cite{hall1950basis,reutenauer93fla}).
  
\begin{theorem}\label{thm1}
Let $A=-\varepsilon^{-1}V$ and  $B=\varepsilon \partial_x^2$  assume they have been discretised following Assumption \ref{ass:spectral}, and stipulate Assumption \ref{ass:potential}. Then all terms $C$ of the Hall basis constructed of letters $A$ and $B$  either vanish (i.e.~$C\equiv 0$) or are $\mathcal{O}(\varepsilon^{-1})$.
\end{theorem}  

Before we proceed with the proof of the theorem, note that all elements of Hall basis \cite{reutenauer93fla} of commutators constructed of letters $A$ and $B$ live in the set
\begin{displaymath}
  \GG{G}=\left\{\sum_{k=0}^K y_k(x)\partial_x^k\,: K\in\BB{Z}_+, y_0,\ldots,y_K\in\CC{C}^\infty_{\mathrm{loc}}(\BB{R})\right\},
\end{displaymath}
by applying the product rule (for differentiation). For example,
\begin{align*}
&[B,A]=-[\partial_x^2,V]=-\left(V^{(2)}+2V^{(1)}\partial_x+V\partial_x^2-V\partial_x^2\right)=-V^{(2)}-2V^{(1)}\partial_x\\
&[[B,A],A]=\varepsilon^{-1}[[\partial_x^2,V],V]=\varepsilon^{-1}2(V^{(1)})^2\\
&[[B,A],B]=-\varepsilon[[\partial_x^2,V],\partial_x^2]= \varepsilon\left(V^{(4)}+4V^{(3)}\partial_x+4V^{(2)}\partial_x^2\right),\\
&[[[B,A],A],A]=0,
\end{align*}
where $V^{(k)}=\partial_x^k V$. We define the height of the commutator $C$ as the largest index of non-zero coefficient $y_K(x)$:
$$
{\rm ht}(C)={\rm ht}\!\left(\sum_{k=0}^K y_k(x)\partial_x^k\right)=K, \text{ where } y_K(x)\not\equiv 0,
$$
One can observe, that ${\rm ht}(A)= 0$, ${\rm ht}(B)=2$, ${\rm ht}([B,A])=1$, ${\rm ht}([[B,A],A])=0$ and ${\rm ht}([[B,A],B])=2$.

In the proof we will also refer to the formula elaborated in \cite{bader14eas}
\begin{align}\nonumber
	\left[ \sum_{i=0}^n  f_i(x)\partial^i_x , \sum_{j=0}^m g_j(x)\partial^j_x \right]
	=&
		\sum_{i=0}^n\sum_{j=0}^m \sum_{\ell=0}^i \binom{i}{\ell} f_i(x)  \left(\partial^{i-\ell}_x g_j(x)\right)\partial^{\ell+j}_x \\
		\label{eq:fullcommutator}
	&
		\mbox{}- \sum_{j=0}^m\sum_{i=0}^n \sum_{\ell=0}^j \binom{j}{\ell} g_j(x)  \left(\partial^{j-\ell}_x f_i(x)\right)\partial^{\ell+i}_x.
\end{align}

\begin{proof}(of Theorem \ref{thm1})
Let us assume, that a certain non-zero commutator $C$ in Hall basis is built of $N_A$ letters $A$ and $N_B$ letters $B$. We  show by induction on $N_A + N_B$, that 
\begin{equation}\label{height}
{\rm ht}(C) \leq N_B-N_A+1.
\end{equation}
The cases in which $N_A + N_B = 1$ are obtained explicitly as ${\rm ht}(A)=0$ and ${\rm ht}(B)=2$, thus \eqref{height} is satisfied for the generators of the free Lie algebra. Now let us assume that a given non-zero commutator $C$ satisfies (\ref{height}), so can be written as
$$
C = \sum_{k=0}^{K} y_k(x) \partial_x^k,
$$
where $0 \leq K \leq  N_B - N_A + 1$ and $y_K \not\equiv 0$. Then by \eqref{eq:fullcommutator},
\begin{align*}
[A,C] &= \varepsilon^{-1} \sum_{k=0}^{K-1} \left(\sum_{j=k}^K \binom{j}{k} y_j(x)  V^{(j-k)}(x)\right) \partial_x^k, \\
[B,C] &= \varepsilon \sum_{k=0}^K y_k''(x) \partial_x^k + 2y_k'(x) \partial_x^{k+1},
\end{align*}
Therefore, ignoring the cases where these commutators vanish identically, we see that \eqref{height} is satisfied for $[A,C]$ and $[B,C]$ by the inductive hypothesis. This, in fact, completes the induction step for the entire Hall basis, because any commutator in the Hall basis can be written as a linear combination of words of the form $$[a_1,[a_2,[\dots,[a_{n-1},a_{n}]\dots]]],$$ where $a_k \in \{A,B\}$ for all $k$, by the Jacobi identity (this is known as the {\em Dynkin basis\/}).

Next we show that every non-zero commutator $C$ in the Hall basis  scales like $\O{\varepsilon^{-1}}$. Indeed, when $C$ is made up of $N_A$ letters $A=-\varepsilon^{-1}V$ and $N_B$ letters $B=\varepsilon \partial_x^2$, the linearity of commutators implies the equality of commutators $C$ and $\varepsilon^{N_B-N_A}\bar{C}$, where $\bar{C}$ has the same structure as $C$, but with $\bar{A}=-V$ and $\bar{B}=\partial_x^2$ instead of $A$ and $B$. Obviously ${\rm ht}(C)={\rm ht}(\bar{C})$. Now by Assumption \ref{ass:spectral}, $\partial_x$ scales like $\varepsilon^{-1}$ after discretisation and by Assumption \ref{ass:potential} we have that all variable coefficients $y_k$ lie in $\CC{C}^\infty_{\mathrm{loc}}(\BB{R})$ (so all derivatives are locally bounded). Therefore $\bar{C} = \mathcal{O}(\varepsilon^{-{\rm ht}(C)})$. Since for non-zero $C$, we have ${\rm ht}(C) \leq N_B-N_A+1$, we conclude that,  
$$
C = \varepsilon^{N_B-N_A} \bar{C} = \mathcal{O}(\varepsilon^{N_B-A_N-{\rm ht}(\bar{C})}) =\mathcal{O}(\varepsilon^{N_B - N_A -(N_B-N_A+1)}) = \mathcal{O}(\varepsilon^{-1}),
$$
which concludes the proof of the theorem.
\end{proof}

An immediate consequence of Theorem \ref{thm1} and \eqref{sBCH} is that
$$
 \ee^{\frac12 \tau A}\ee^{\tau B}\ee^{\frac12 \tau A}=\ee^{\tau (A+ B) + \O{h^3\varepsilon^{-1}}}=\ee^{\tau (A+ B)} + \O{h^3\varepsilon^{-1}}.
$$  
This means that taking the time step size $h=\O{\varepsilon}$ in the Strang splitting (\ref{eq:1.4}) yields a local truncation error of $\mathcal{O}(h^2)$ or equivalently, $\mathcal{O}(\varepsilon^{2})$. However, a time step $h=\O{\varepsilon}$ is overly expensive. If instead, one took a more reasonable $h=\O{\varepsilon^{1/2}}$, then the local truncation error is effectively $\mathcal{O}(h)$ or equivalently, $\mathcal{O}(\varepsilon^{1/2})$. In summary, unless the time step is unacceptably reduced, the effective error of the Strang splitting is larger than that suggested by an analysis which ignores the smallness of $\varepsilon$.

\subsection{Symmetric Zassenhaus splittings}

This order reduction for the Strang splitting in the case of Hamiltonians in a semi-classical setting motivates the quest for higher order splittings. A systematic approach is to calculate higher order {\em symmetric Zassenhaus splittings\/}, first proposed in \cite{bader14eas}. Using this methodology we will derive two splittings for the solution operator $\exp(\tau (A+ B))$ where $A=-\varepsilon^{-1}V$, $B=\varepsilon \partial_x^2$ and $\tau=\ii h$, of order $\O{h^3\varepsilon^{-1}}$ and $\O{h^5\varepsilon^{-1}}$ respectively, in the family of symmetric Zassenhaus splittings. 
\begin{enumerate}
\item
To derive the first symmetric Zassenhaus splitting, we apply the sBCH formula in the following way.
\begin{equation}\label{first_level_of_sZS}
\ee^{-\frac12 \tau A}\ee^{\tau A+\tau B}\ee^{-\frac12 \tau A}=\ee^{\mathrm{sBCH}(-\tau A,\tau A+\tau B)}
\end{equation} 
where
\begin{align}\label{sBCH_first}
    & \mathrm{sBCH} (-\tau A,\tau A+\tau B)=\\ \nonumber
     &=\tau B+\tau^{3}\frac{1}{24}[[ B,A] A]+\tau^{3}\frac{1}{12}[[ B,A], B]\\\nonumber
     &-\tau^{5}\frac{1}{720}[[[[B,A],A],B],B]-\tau^{5}\frac{1}{720}[[[[B,A],B],B],B]\\\nonumber
    &-\tau^{5}\frac{1}{480}[[[B,A],A],[B,A]]-\tau^{5}\frac{1}{240}[[[B,A],B],[B,A]]+\O{h^7\varepsilon^{-1}}.\nonumber
\end{align}

Substituting \eqref{sBCH_first} into \eqref{first_level_of_sZS} results in the first symmetric Zassenhauss splitting, which coincides with Strang splitting,
\begin{align}\label{first_level_of_sZS_a}
\ee^{\tau (A+B)}&=\ee^{\frac12 \tau A} \ee^{\mathrm{sBCH}(-\tau A,\tau A+\tau B)} \ee^{\frac12 \tau A}\\ \label{first_level_of_sZS_b}
&=\ee^{\frac12 \tau A} \ee^{\tau B} \ee^{\frac12 \tau A}+\O{h^3\varepsilon^{-1}}.\nonumber
\end{align}
\item
To derive the second symmetric Zassenhaus splitting, we split the inner term of \R{first_level_of_sZS_a} by the same approach as above, that is
\begin{align*}
\ee^{-\frac12 \tau B}\ee^{\mathrm{sBCH}(-\tau A,\tau A+\tau B)}\ee^{-\frac12 \tau B}&=\ee^{\mathrm{sBCH}(-\tau B, \mathrm{sBCH}(-\tau A,\tau A+\tau B))}
\end{align*}
which leads to,
\begin{equation}\label{second_level_of_sZS}
							\ee^{\tau A+\tau B}= \ee^{\frac12 \tau A}e^{\frac12 \tau B} \ee^{\mathrm{sBCH}(-\tau B, \mathrm{sBCH}(-\tau A,\tau A+\tau B))}\ee^{\frac12 \tau B}e^{\frac12 \tau A},
\end{equation}
where
\begin{align*}
&\mathrm{sBCH}(-\tau B,\mathrm{sBCH}( -\tau A,\tau B+\tau A))\\
    &=\frac{1}{24}\tau^3[[B,A],A]+\frac{1}{12}\tau^3[[B,A],B]\\
&-\frac{19}{2880}\tau^5[[[[B,A],A], B], B]-\frac{17}{1440}\tau^5[[[[B,A],B],B], B]\\
&-\tau^{5}\frac{1}{480}[[[B,A],A],[B,A]]-\tau^{5}\frac{1}{240}[[[B,A],B],[B,A]]+\O{h^7\varepsilon^{-1}}.
    \end{align*}

Observe that by Theorem \ref{thm1}, the first two commutators (which involve three letters) scale like $\O{h^3\varepsilon^{-1}}$ and the remainder scales like $\O{h^5\varepsilon^{-1}}$. Therefore, these first two terms are what will appear in this Zassenhaus splitting. However, the commutator,
\begin{align*}
[[B,A],B] = [[\varepsilon\partial_x^2,-\varepsilon^{-1}V],\varepsilon\partial_x^2]&=\varepsilon\left(V^{(4)}+4V^{(3)}\partial_x+4V^{(2)}\partial_x^2\right),
\end{align*}
will not be skew-Hermitian after discretisation (which would result in loss of unitarity of the method), and therefore cannot be substituted into (\ref{second_level_of_sZS}). For this reason, as proposed in \cite{bader14eas}, we use a substitution rule of the following kind:
$$
y(x)\partial_x=-\frac{1}{2}\left[\int_{x_0}^xy(s){\rm d}s\right]\partial_x^2-\frac{1}{2}\partial_xy(x)+\frac{1}{2}\partial_x^2\left[\int_{x_0}^xy(s) {\rm d}s \, \cdot\right],
$$
and obtain terms that remain skew-Hermitian after discretisation:
\begin{align*}
&\mathrm{sBCH}(-\tau B,\mathrm{sBCH}(-\tau A,\tau B+\tau A))\\
    &=\tau^3\varepsilon^{-1}\frac{1}{12}(V^{(1)})^2+\tau^3\varepsilon\frac{1}{12}V^{(4)}+\tau^3\varepsilon\frac{1}{3}V^{(3)}\partial_x+\tau^3\varepsilon\frac{1}{3}V^{(2)}\partial_x^2+\O{h^5\varepsilon^{-1}}\\
    =&\tau^3\varepsilon^{-1}\frac{1}{12}(V^{(1)})^2+\frac{1}{6}\tau^3\varepsilon\underbrace{\left\{V^{(2)}\partial_x^2+\partial_x^2\left[V^{(2)}\cdot\right]\right\}}_{\O{\varepsilon^{-2}}}- \frac{1}{12}\tau^3\varepsilon V^{(4)}+\O{h^5\varepsilon^{-1}}.
\end{align*}
In the final form of the splitting (\ref{second_level_of_sZS}) the small $\O{h^3\varepsilon}$ term involving $V^{(4)}$ can be discarded.
\end{enumerate}

Summing up these two derivations, we have the splittings,
\begin{equation}
  \label{eq:1.5}
  u^{k+1}(x)=\ee^{\mathcal{R}_0} \ee^{2\mathcal{R}_1}\ee^{\mathcal{R}_0} u^k(x) +\O{h^3\varepsilon^{-1}}
\end{equation}
and
\begin{equation}
  \label{eq:1.6}
  u^{k+1}(x)=\ee^{\mathcal{R}_0} \ee^{\mathcal{R}_1}\ee^{2\mathcal{R}_2}\ee^{\mathcal{R}_1}\ee^{\mathcal{R}_0} u^k(x) +\O{h^5 \varepsilon^{-1}},
\end{equation}
where, letting $\tau=\ii h$,
\begin{Eqnarray*}
  \mathcal{R}_0&=&-\frac12 \tau \varepsilon^{-1} V,\\
  \mathcal{R}_1&=&\frac12 \tau\varepsilon \partial_x^2,\\
  \mathcal{R}_2&=&\black\frac{1}{12}\tau^3\varepsilon \left\{\partial_x^2[V^{(2)}\,\cdot\,]+V^{(2)}\partial_x^2\right\} +\frac{1}{24} \tau^3\varepsilon^{-1} (V^{(1)})^2.
\end{Eqnarray*}
Note that $\mathcal{R}_0=\O{h\varepsilon^{-1}}$, $\mathcal{R}_1=\O{h\varepsilon^{-1}}$, $\mathcal{R}_2=\O{h^3\varepsilon^{-1}}$.

It is also possible to derive even higher order methods, such as
\begin{equation}
  u^{n+1}(x)=\ee^{\mathcal{R}_0} \ee^{\mathcal{R}_1}\ee^{\mathcal{R}_2}\ee^{2\mathcal{R}_3}\ee^{\mathcal{R}_2}\ee^{\mathcal{R}_1}\ee^{\mathcal{R}_0} u^n(x) +\O{h^7 \varepsilon^{-1}},
  \end{equation}
where
\begin{eqnarray*}
  \mathcal{R}_3&=&-\frac{1}{120} \tau^5 \varepsilon^{-1} V^{(2)}(V^{(1)})^2+\frac{1}{24}\tau^3\varepsilon V^{(4)}\\
  &&+\frac{1}{120} \tau^5 \varepsilon \left\{\partial_x^2\left[\left(7(V^{(2)})^2+V^{(3)}V^{(1)}\right)\, \cdot \right]+\left(7(V^{(2)})^2+V^{(3)}V^{(1)}\right)\partial_x^2\right\}\\
  &&+\frac{1}{60} \tau^5 \varepsilon^{-3} \left\{\partial_x^4\left[V^{(4)}\,\cdot\,\right]+V^{(4)}\partial_x^4\right\}.
  \end{eqnarray*}
Note that $\mathcal{R}_3 = \O{h^5 \varepsilon^{-1}}$. We refer the reader to \cite{singh} for discussion of deriving such higher order methods via a sequence of skew-Hermitian operators $\mathcal{R}_0,\mathcal{R}_1,\ldots$. Our new analysis encapsulated in Theorem \ref{thm1} shows that each term $\mathcal{R}_\ell$ is actually of size $\O{h^{2\ell -1}\varepsilon^{-1}}$ for $\ell = 1,2,\ldots$. In Section \ref{sec:conclusion}, we will discuss how to go about computing $\ee^{\mathcal{R}_\ell}$ for each $\ell$.

\setcounter{equation}{0}
\setcounter{figure}{0}
\section{Orthonormal systems and free Schr\"odinger evolutions}\label{sec:orthoFSE}

\subsection{Orthogonal systems with tridiagonal differentiation matrices}\label{subsec:orthosystems}

Solving \R{eq:1.1} by spectral methods based upon  symmetric Zassenhaus splittings \R{eq:1.5} or \R{eq:1.6} involves three ingredients: the splitting itself into $\mathcal{R}_0,\mathcal{R}_1,\mathcal{R}_2,\ldots$, the choice of spectral basis, and the means to compute the exponentials $\ee^{\mathcal{R}_\ell}$. The generalisation of each to the new setting requires new ideas and substantial effort. In this subsection we are concerned with the choice of the spectral basis.

We seek a set $\Phi=\{\varphi_n\}_{n=0}^\infty$ which forms an orthonormal basis of $\CC{L}_2(\BB{R})$ -- this means that any $f\in\CC{L}_2(\BB{R})$ can be expanded in the form
\begin{displaymath}
  f(x)=\sum_{n=0}^\infty \hat{f}_n \varphi_n(x),\qquad\mbox{where}\qquad \hat{f}_n=\int_{-\infty}^\infty f(x)\overline{\varphi_n(x)}\D x,\quad n\in\BB{Z}_+.
\end{displaymath}
For the time being we require the $\varphi_n$s to be real, although this will be lifted as necessary (with suitable changes). In addition we require that $\Phi$ has a \emph{tridiagonal differentiation matrix} (which, it is easy to prove, must be skew-symmetric),
\begin{equation}
  \label{eq:2.1}
  \varphi_n'=-b_{n-1} \varphi_{n-1} +b_n \varphi_{n+1},\qquad n\in\BB{Z}_+,
\end{equation}
where $b_{-1}=0$ and $b_n>0$, $n\in\BB{Z}_+$. This makes both computation and analysis considerably easier.

A comprehensive theory of such orthogonal systems has been developed in \cite{iserles19oss,iserles2020differential}. The main issue, making \R{eq:2.1} compatible with orthonormality, can be explicated by considering Fourier transforms of the $\varphi_n$s. Specifically, let $w(\xi)\D\xi$ be a Borel measure over $\BB{R}$ and its Radon--Nikodym derivative $w\geq0$ be absolutely continuous and even. Furthermore assume that all the moments of this measure are finite. Such measure generates a system of {\em orthonormal polynomials\/} $\{p_n\}_{n=0}^\infty$,
\begin{displaymath}
  \int_{-\infty}^\infty p_n(\xi)p_m(\xi) w(\xi)\D\xi=0,\quad m\neq n,\qquad \int_{-\infty}^\infty p_n^2(\xi)w(\xi)\D\xi=1.
\end{displaymath}
Then the scaled inverse Fourier transform,
\begin{equation}
  \label{eq:2.2}
  \varphi_n(x)=\frac{(-\ii)^n}{\sqrt{2\pi}} \int_{-\infty}^\infty p_n(\xi) g(\xi) \ee^{\ii x\xi}\D\xi,\qquad n\in\BB{Z}_+,
\end{equation}
 where $g$ is any function satisfying $|g(\xi)|^2 = w(\xi)$, forms an orthonormal system on the real line which satisfies \eqref{eq:2.1}. Note that this system is real-valued if and only if $g$ has even real part and odd imaginary part, for example $g(\xi) = \sqrt{w(\xi)}$. The constants $b_n$ in \R{eq:2.1} are inherited from the recurrence relation for orthonormal polynomials,
\begin{displaymath}
  b_n p_{n+1}(\xi)=\xi p_n(\xi)-b_{n-1}p_{n-1}(\xi),\qquad n\in\BB{Z}_+.
\end{displaymath}

The orthonormal system given by \R{eq:2.2} need not be dense in $\BB{R}$ -- as a matter of fact, it is dense in the {\em Paley--Wiener space\/} $\mathcal{PW}_{\mbox{supp}(w)}(\BB{R}) \subseteq \CC{L}_2(\BB{R})$ which is the space of $L_2(\BB{R})$ functions whose Fourier transforms vanish outside of the support of $w$. Therefore, the system is a basis of $\CC{L}_2(\BB{R})$ if and only if the weight function $w$ is positive on the whole real line. 

Complete orthonormal bases can be formed also from polynomials $P=\{p_n\}_{n=0}^\infty$ orthogonal on the half-line $[0,\infty)$ \cite{iserles20for}, e.g.\ the Laguerre polynomials whose orthogonality measure is $\ee^{-\xi}\D\xi$, $\xi\geq0$: The representation \R{eq:2.2} survives intact but, to render the system dense in $\CC{L}_2(\BB{R})$, we need to complement $P$ with orthogonal polynomials with respect to the mirror image of $w$ in the left half-line, $w(-\xi)\D\xi$ for $\xi\leq0$. The new system $\Phi$ is enumerated by $n\in\BB{Z}$ and in place of \R{eq:2.1} we have
\begin{displaymath}
  \varphi_n'=-b_{n-1}\varphi_{n-1}+\ii c_n\varphi_n+b_n\varphi_{n+1},\qquad n\in\BB{Z},
\end{displaymath}
with $b_n>0$, $n\neq0$, $b_0=0$ and real $c_n$ -- note that the new differentiation matrix is skew-Hermitian. 

\subsection{Free Schr\"odinger evolutions}

Given an orthonormal system $\Phi$ on the real line, we denote by $\psi_n$, $n\in\BB{Z}_+$, the solution of the free Schr\"odinger equation \R{eq:1.3} with the initial condition $\varphi_n$ -- in other words,
\begin{equation}
  \label{eq:2.3}
  \frac{\partial \psi_n}{\partial t}=-\ii\varepsilon \frac{\partial^2 \psi_n}{\partial x^2},\qquad \psi_n(x,0)=\varphi_n(x),\; x\in\BB{R}.
\end{equation}
We call $\Psi(t)=\{\psi_n(\cdot,t)\}_{n=0}^\infty$ the {\em free Schr\"odinger evolution (FSE)\/} of  $\Phi$.

The exact solution of \R{eq:2.3} via the Fourier transform is well known and can be easily verified by direct differentiation:
\begin{equation}
  \label{eq:2.4}
  \psi_n(x,t)=\frac{1}{\sqrt{2\pi}} \int_{-\infty}^\infty \hat{\varphi}_n(\eta) \ee^{\ii\eta^2\varepsilon t+\ii\eta x}\D\eta,
\end{equation}
where
\begin{displaymath}
  \hat{\varphi}_n(\eta)=\frac{1}{\sqrt{2\pi}} \int_{-\infty}^\infty \varphi_n(\xi) \ee^{-\ii\eta\xi}\D\xi
\end{displaymath}
is the familiar Fourier transform of $\varphi_n$. 

On the face of it, our job is done: any mention of the phrase ``Fourier transform'' elicits from a numerical analyst the instinctive response ``Fast Fourier Transform!''. This, however, is somewhat rash. An FFT computes rapidly the discrete Fourier transform which, in turn, is a very precise (at any rate, for very smooth functions) approximation of the Fourier transform {\em of a periodic function in a compact interval,\/} while our setting is the entire real line. One possibility is to clip the real line, approximating it by a sufficiently large interval and disregarding the Gibbs effect at the endpoints. This immediately begs the question ``how large'' which, while not beyond the ken of numerical reasoning, presents its own challenges. In this paper we adopt an alternative -- and arguably more effective -- point of view, seeking the {\em exact\/} solution of \R{eq:2.4} for specific orthonormal systems $\Phi$. While this approach cannot be expected to apply to each and every $\Phi$ consistent with the setting of Subsection~2.1, it does so with the two most interesting orthonormal systems: Hermite functions and Malmquist--Takenaka functions.

Once FSEs $\Psi(t)$ are known, the solution of the free Schr\"odinger equation \R{eq:1.3} with the initial condition $u(x,kh)$ proceeds as follows: The function $u(x,kh)$ is expanded in the orthonormal basis $\Phi$,
\begin{equation}
  \label{eq:2.5}
  u(x,kh)\approx\sum_{n=0}^N \hat{u}_n \varphi_n(x)
\end{equation}
for a sufficiently large truncation parameter $N$. Having done so, linearity of \R{eq:1.3} implies that
\begin{equation}
  \label{eq:2.6}
  u(x,(k+1)h)\approx \sum_{n=0}^N \hat{u}_n \psi_n(x,h).
\end{equation}
We get the coefficients for free because they do not change --- it is the basis which changes. The choice of $N$ is governed by approximation properties of the spectral basis, and its ability to approximate spatial oscillations of frequency $\O{\varepsilon^{-1}}$ as discussed in the introduction.

Indeed, orthonormal systems are not all of equal value: more specifically, they can approximate functions at different speeds. While standard spectral methods on a torus are known to converge (for analytic functions) at an exponential speed, equivalent theory does not exist yet on the real line. Recalling from Section~1 that solutions of \R{eq:1.1} are typically composed of wave packets, it is instructive to enquire how well  different orthonormal systems approximate wave packets. This is investigated in \cite{iserles20awp} for the two families $\Phi$ described in the sequel: in both cases we can prove exponential convergence to any set error tolerance.

We note for further reference that the computation of \R{eq:2.6} (once $N$ and $h$ have been appropriately chosen) requires both the knowledge of $\Psi(h)$ and the means to evaluate an expansion as in \R{eq:2.5}.

\begin{theorem}
  \label{thm:1}
  Let $\Phi$ be as in \eqref{eq:2.2}. Then the functions,
  \begin{equation}
    \label{eq:2.7}
    \psi_n(x,t)=\frac{(-\ii)^n}{\sqrt{2\pi}} \int_{-\infty}^\infty p_n(\xi)g(\xi) \ee^{\ii x\xi+\ii \varepsilon t\xi^2}\D\xi,\qquad n\in\BB{Z}_+,
  \end{equation}
  where $\{p_n\}_{n=0}^\infty$ is the system of orthonormal polynomials with respect to the measure $|g(\xi)|^2\D\xi$, satisfies \eqref{eq:2.3} (in particular $\psi_n(x,0) = \varphi_n(x)$) and for all $t$ is itself an orthonormal basis of $\CC{L}_2(\BB{R})$ satisfying, 
  \begin{equation}
    \label{eq:2.8}
     \frac{\partial\psi_n(x,t)}{\partial x}=-b_{n-1}\psi_{n-1}(x,t)+b_n\psi_{n+1}(x,t),\qquad n\in\BB{Z}_+,
  \end{equation}
  where $\{b_n\}_{n\in\bb{Z}_+}$ are the same constants as in \eqref{eq:2.1}.
\end{theorem}

\begin{proof}
   Differentiating under the integral sign with respect to $x$ twice and $t$ once demonstrates that $\psi_n(x,t)$ satisfies the free Schr\"odinger equation \eqref{eq:2.3}, and it is clear that setting $t = 0$ in this formula yields $\varphi_n(x)$.
   
   To show that $\psi_n$ is an orthonormal system satisfying \eqref{eq:2.8}, note that
    \begin{equation}
   \left| g(\xi) \ee^{\ii \varepsilon t \xi^2} \right|^2 = |g(\xi)|^2 = w(\xi),
   \end{equation}
   so these functions still come under the framework of \eqref{eq:2.2}, with exactly the same polynomials $\{p_n\}_{n\in\bb{Z}_+}$, but with the function $g(\xi) \ee^{\ii \varepsilon t \xi^2}$ in place of $g(\xi)$.

\end{proof}

\subsection{Re-expanding an FSE expansion in the original basis}

Suppose that have an expansion the FSE basis $\Psi(t) = \{ \psi_{n}\}_{n=0}^\infty$,
\begin{equation}
  u(x,t) = \sum_{n=0}^\infty a_n \psi_n(x,t), 
   \end{equation}
and we wish to re-expand this basis in terms of the original basis $\Phi (= \Psi(0))$ for each $t$. Let us consider time-dependent coefficients $\alpha_n(t)$ satisfying
\begin{equation}
  u(x,t) = \sum_{n=0}^\infty \alpha_n(t) \varphi_n(x).
  \end{equation}
The relationship between $\boldsymbol{\alpha}(t)$ and $\boldsymbol{a}$ is simple when considered in terms of the polynomial basis $P$. Indeed, the relationship is given by
\begin{equation}
   \sum_{n=0}^\infty (-\ii)^n \alpha_n(t) p_n(\xi) = \sum_{n=0}^\infty (-\ii)^n a_n p_n(\xi) \ee^{\ii t \xi^2},
  \end{equation}
where the expansions are convergent in the space $\CC{L}_2(\BB{R},w(\xi)\mathrm{d}\xi)$. Writing this in terms of operators acting on the vectors $\boldsymbol{a},\boldsymbol{\alpha}(t) \in \ell_2$, we have
\begin{equation}
  T S\boldsymbol{\alpha}(t) =   M(t)TS\boldsymbol{a},
  \end{equation}
where $S : \ell_2 \to \ell_2$ simply multiplies $n$th component of a sequence by $(-\ii)^n$, $T : \ell_2 \to \mathrm{L}_2(\BB{R},w(\xi)\mathrm{d}\xi)$ is the synthesis operator for the basis $P$ (AKA coefficients-to-values operator), and $M(t) : \mathrm{L}_2(\BB{R},w(\xi)\mathrm{d}\xi) \to \mathrm{L}_2(\BB{R},w(\xi)\mathrm{d}\xi)$ multiplies functions by $\exp\left( \ii t \xi^2\right)$. Note that since $P$ is an orthonormal basis for $\mathrm{L}_2(\BB{R},w(\xi)\mathrm{d}\xi)$, we have that $T$ is a unitary operator (the inverse, $T^*$ is usually called the analysis operator or values-to-coefficients operator). $S$ is also clearly unitary, so we can invert operators to find,
\begin{equation}
    \boldsymbol{\alpha}(t) =   S^* T^* M(t)TS\boldsymbol{a}.
  \end{equation}
Since $M(t)$ is unitary, we see that as expected, the operation sending $\boldsymbol{a}$ to $ \boldsymbol{\alpha}(t)$ is unitary overall.

Now, let us project these equations onto the first $N+1$ terms of $\Phi$. We obtain,
\begin{equation}
  \boldsymbol{\alpha}^{[N]}(t) =   S_N^* T_N^* M_N(t)T_N S_N\boldsymbol{a}^{[N]},
\end{equation}
where $\boldsymbol{\alpha}^{[N]}(t), \boldsymbol{a}^{[N]} \in \BB{C}^{N+1}$, and $S_N,T_N,M_N(t) : \BB{C}^{N+1} \to\BB{C}^{N+1}$. These discretised operators are, $S_N = \mathrm{diag}((-\ii)^n)_{n=0}^N$, $M_N(t) = \mathrm{diag}(\exp( \ii t \xi_k^2))_{k=0}^N$, and
\begin{equation}
  T_N = \begin{pmatrix}
      \sqrt{w_0} p_0(\xi_0) & \sqrt{w_0} p_1(\xi_0) & \cdots & \sqrt{w_0} p_N(\xi_0) \\
      \sqrt{w_1} p_0(\xi_1) & \sqrt{w_1} p_1(\xi_1) & \cdots & \sqrt{w_1} p_N(\xi_1) \\
      \vdots & \vdots & & \vdots \\
      \sqrt{w_N} p_0(\xi_N) & \sqrt{w_N} p_1(\xi_N) & \cdots & \sqrt{w_N} p_N(\xi_N) \\
    \end{pmatrix},
  \end{equation}
where $w_0,\ldots, w_N, \xi_0, \ldots, \xi_N$ are Gauss quadrature weights and nodes (respectively) for the measure $w(\xi)\mathrm{d}\xi$. First, note that the unitarity of the operators has been preserved by this discretisation. Second, note that the unitary matrix $T_N$ and the nodes $\{\xi_k\}_{k=0}^N$ can be computed rapidly and stably by computing the eigendecomposition of the Jacobi matrix for the orthonormal polynomials $P$, as in the Golub--Welsch algorithm \cite{golub1969calculation}. However, if $P$ is an orthonormal polynomial basis which enjoys fast transforms from coefficients to values and back, and fast computation of Gaussian quadrature nodes (Jacobi polynomials, for example \cite{townsend2018fast}) then such algorithms can be used in place of the generic Golub--Welsch approach.

\setcounter{equation}{0}
\setcounter{figure}{0}
\section{Examples of orthonormal systems}

In this section we describe two systems $\Phi$ and their free Schr\"odinger evolutions $\Psi(t)$.

\subsection{Hermite functions}

Hermite functions
\begin{equation}
  \label{eq:3.1}
  \varphi_n(x)=\frac{1}{{(2^nn!\pi^{1/2})}^{1/2}} \CC{H}_n(x)\ee^{-x^2/2},\qquad n\in\BB{Z}_+,
\end{equation}
where $\CC{H}_n$ is the $n$th {\em Hermite polynomial,\/} are eigenfunctions of the Fourier transform,
\begin{equation}
  \label{eq:3.2}
  \frac{1}{\sqrt{2\pi}}\int_{-\infty}^\infty \varphi_n(\xi) \ee^{\ii x\xi}\D\xi=\ii^n\varphi_n(x),\qquad x\in\BB{R},\quad n\in\BB{Z}_+.
\end{equation}
Their orthonormality in $\CC{L}_2(\BB{R})$ follows from that of the familiar Hermite polynomials \cite[18.3]{dlmf} in $\CC{L}_2(\BB{R};\ee^{-\xi^2})$, they obey the differential recurrence relation \eqref{eq:2.2} with $b_n=\sqrt{n/2}$ and the Cram\'er inequality $|\varphi_n(x)|\leq\pi^{-1/4}$, $x\in\BB{R}$.\footnote{They should not be confused with Hermite functions from \cite[p.~84]{ismail20uop}.}

To derive the FSE $\Psi=\{\psi_n\}_{n=0}^\infty$ we assume the atomistic setting $\varepsilon=1$: to translate to semiclassical setting, we will replace $t$ by $\varepsilon t$ in the final formula. Our starting point is the standard generating function for Hermite polynomials,
\begin{displaymath}
  \sum_{n=0}^\infty \frac{\CC{H}_n(x)}{n!} z^n=\ee^{2xz-z^2}
\end{displaymath}
\cite[18.12.15]{dlmf}. It now follows from \R{eq:3.1} that
\begin{displaymath}
  \pi^{1/4}\ee^{x^2/2}\sum_{n=0}^\infty \frac{\varphi_n(x)}{\sqrt{n!}} (2^{1/2}z)^n=\ee^{2xz-z^2}
\end{displaymath}
or, replacing $z\rightarrow 2^{-1/2}z$,
\begin{displaymath}
  \sum_{n=0}^\infty \frac{\varphi_n(x)}{\sqrt{n!}} z^n=\pi^{-1/4}\exp\!\left(-\frac{x^2}{2}+2^{1/2}xz-\frac{z^2}{2}\right)\!.
\end{displaymath}
It now follows from \R{eq:2.4} and \R{eq:3.2} that 
\begin{Eqnarray*}
  \sum_{n=0}^\infty \frac{\psi_n(x,t)}{\sqrt{n!}} (\ii z)^n&=&\frac{1}{\sqrt{2\pi}} \sum_{n=0}^\infty \frac{z^n}{\sqrt{n!}} \int_{-\infty}^\infty \varphi_n(\xi) \ee^{\ii(\xi^2 t+\xi x)}\D\xi\\
  &=&\frac{1}{\sqrt{2\pi}} \int_{-\infty}^\infty \left[\sum_{n=0}^\infty \frac{\varphi_n(\xi)}{\sqrt{n!}} z^n\right] \! \ee^{\ii(\xi^2t+\xi x)}\D\xi \\
  &=&\frac{1}{2^{1/2}\pi^{3/4}} \int_{-\infty}^\infty \exp\!\left(-\frac12\xi^2+2^{1/2}\xi z-\frac12 z^2+\ii\xi^2 t+\ii\xi x\right)\!\D\xi\\
  &=&\frac{1}{\pi^{1/4}(1-2\ii t)^{1/2}} \exp\!\left(-\frac{z^2+2^{3/2}\ii xz-x^2+2\ii tz^2}{2(2\ii t-1)}\right)\!.
\end{Eqnarray*}
We conclude that
\begin{displaymath}
  \sum_{n=0}^\infty \frac{\psi_n(x,t)}{\sqrt{n!}} (\ii z)^n=\frac{1}{\pi^{1/4}(1-2\ii t)^{1/2}} \exp\!\left(\frac{x^2}{2(2\ii t-1)}\right) \exp\!\left(-\frac{2^{1/2}\ii xz}{2\ii t-1} -\frac12 \frac{2\ii t+1}{2\ii t-1}z^2\right)\!.
\end{displaymath}
Set
\begin{displaymath}
  X=-\frac{x}{(1+4t^2)^{1/2}},\qquad Z=\frac{1}{2^{1/2}} \left(\frac{2\ii t+1}{2\ii t-1}\right)^{\!1/2}z,
\end{displaymath}
which satisfy,
\begin{displaymath}
  2XZ-Z^2=-\frac{2^{1/2}\ii xz}{2\ii t-1} -\frac12 \frac{2\ii t+1}{2\ii t-1}z^2,
\end{displaymath}
and we deduce, using again the generating function for Hermite polynomials, that
\begin{displaymath}
  \exp\!\left(-\frac{2^{1/2}\ii xz}{2\ii t-1} -\frac12 \frac{2\ii t+1}{2\ii t-1}z^2\right)=\sum_{n=0}^\infty \frac{\CC{H}_n(X)}{n!}Z^n.
\end{displaymath}
All we thus need is to compare the powers of $z$ in 
\begin{Eqnarray*}
  &&\sum_{n=0}^\infty \frac{\psi_n(x,t)}{\sqrt{n!}} (\ii z)^n\\
  &=&\frac{1}{\pi^{1/4}(1-2\ii t)^{1/2}} \exp\!\left(\frac{x^2}{2(2\ii t-1)}\right) \sum_{n=0}^\infty \frac{\CC{H}_n(X)}{n!} Z^n\\
  &=&\frac{1}{\pi^{1/4}(1-2\ii t)^{1/2}} \exp\!\left(\frac{x^2}{2(2\ii t-1)}\right) \sum_{n=0}^\infty \frac{1}{n!} \CC{H}_n\!\left(\!-\frac{x}{(1+4t^2)^{1/2}}\right) \!\!\left(\frac{1}{2^{1/2}} \left(\frac{2\ii t\!+\!1}{2\ii t\!-\!1}\right)^{\!1/2}\!z\!\right)^{\!n}\!\!\!. 
\end{Eqnarray*}
The outcome is
\begin{displaymath}
  \psi_n(x,t)=\frac{\ii^n}{(2^nn!\pi^{1/2})^{1/2}(1-2\ii t)^{1/2}} \exp\!\left(\frac{x^2}{2(2\ii t-1)}\right)\!\left(\frac{2\ii t+1}{2\ii t-1}\right)^{\!n/2}\! \CC{H}_n\!\left(\frac{x}{(1+4t^2)^{1/2}}\right)\!.
\end{displaymath}
Finally, since
\begin{displaymath}
  \CC{H}_n\!\left(\frac{x}{(1+4t^2)^{1/2}}\right)=(2^nn!)^{1/2}\pi^{1/4} \exp\!\left(\frac{x^2}{2(1+4t^2)}\right)\!\varphi_n\!\left(\frac{x}{(1+4t^2)^{1/2}}\right)\!,
\end{displaymath}
we deduce, restoring the semiclassical setting, that

\begin{lemma}\label{lem:HermiteFSE}
  The explicit form of the Hermite FSE is
  \begin{equation}
    \label{eq:3.3}
    \psi_n(x,t)= \frac{(1+2\ii \varepsilon t)^{n/2}}{(1-2\ii\varepsilon t)^{(n+1)/2}} \exp\!\left(-\frac{\ii t\varepsilon x^2}{1+4\varepsilon^2t^2}\right) \!\varphi_n\!\left(\frac{x}{(1+4\varepsilon^2t^2)^{1/2}}\right)\!.
  \end{equation}
  Moreover, the functions $\psi_n$ are subject to the bound
  \begin{equation}
    \label{eq:3.4}
    |\psi_n(x,t)|\leq \frac{1}{[\pi(1+4\varepsilon^2t^2)]^{1/4}},\qquad t\geq0,\;\; x\in\BB{R}.
  \end{equation}
\end{lemma}

\begin{proof}
 The expression \R{eq:3.3} follows from the preceding analysis, while \R{eq:3.4} is an immediate consequence of the Cram\'er inequality. 
\end{proof}


\begin{figure}
  \centering
  \includegraphics[width=\textwidth, trim=30 30 30 30]{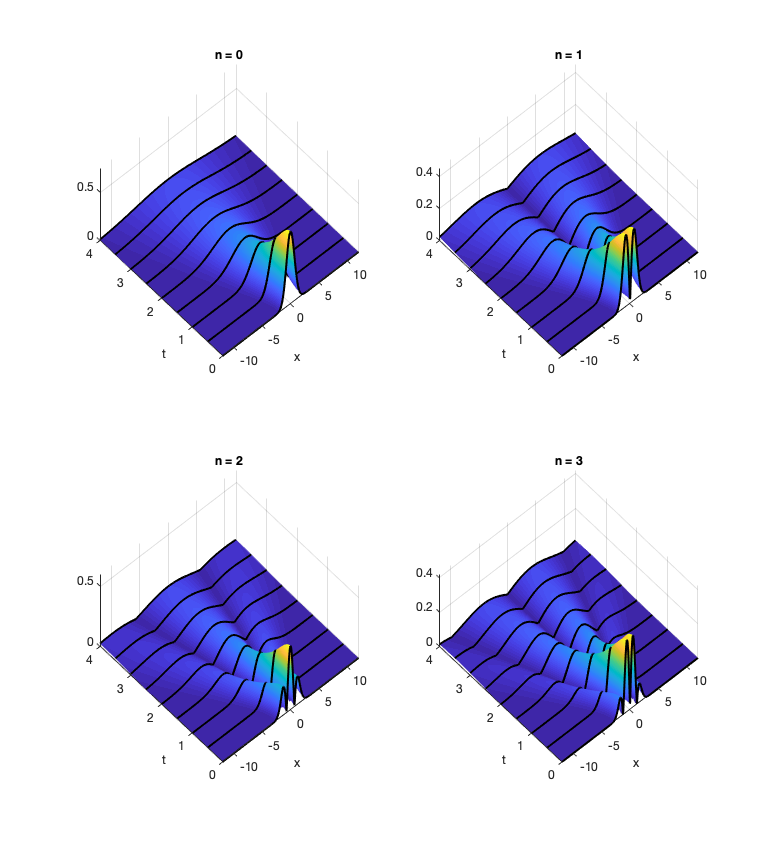}
  \caption{The Hermite FSE: the functions $|\psi_n(x,t)|$ for $n=0,\ldots,3$, $x\in[-12,12]$ and $t\in[0,4]$.}
       \label{fig:3.1}
  \end{figure}

Fig.~\ref{fig:3.1} displays the magnitude of the first six $\psi_n$s. It is evident that they are consistent with the inequality \R{eq:3.4}. There are two facts to bear in mind. Firstly, examining the modulus hides the oscillations in \R{eq:3.3}: in reality, the $\psi_n$s are considerably more violent. Secondly, while the functions $\psi_n$ appear to spread energy and $|\psi_n|$ seems to approach a steady steady, in reality we are interested only is small values of $t$, a single time step, so that $t = h = \O{\varepsilon^{1/2}}$.

An implementation of FSEs based on Hermite functions necessitates in each time step the expansion of the initial value in Hermite functions. There exist powerful algorithms to this end, many based upon the fast multipole algorithm and generalisable to higher spatial dimensions \cite{dutt96fap}.

\begin{lemma}
  The Hermite FSE in Lemma \ref{lem:HermiteFSE} satisfy the three term recurrence,
  \begin{equation}
  x  \psi_n(x,t) = \sqrt{\frac{n}{2}} \left(\frac{1+2\ii\varepsilon t}{1-2\ii\varepsilon t} \right)^{\tfrac12} \psi_{n-1}(x,t) + \sqrt{\frac{n+1}{2}} \left(\frac{1-2\ii\varepsilon t}{1+2\ii\varepsilon t} \right)^{\tfrac12} \psi_{n+1}(x,t).
  \end{equation}
  \end{lemma}

This three term recurrence allows us to evaluate finite expansions in this basis in a stable manner using Clenshaw's algorithm \cite{clenshaw55nsc}.

\subsection{Malmquist--Takenaka functions}

The Malmquist--Takenaka system is a complex-valued rational basis of $\CC{L}_2(\BB{R})$, introduced independently by Malmquist and Takenaka and repeatedly rediscovered: we refer to \cite{iserles20for} for its brief history. It is instructive to introduce them within  the narrative of Subsection~2.1, while extending it to complex-valued bases. The starting point is the {\em Laguerre measure\/} $\ee^{-\xi}\D\xi$, $\xi\geq0$. We can use \R{eq:2.2} to generate an orthonormal system on the real line but this system is not dense in $\CC{L}_2(\BB{R})$. It is a basis of $\mathcal{PW}_{[0,\infty)}(\BB{R})$, of $f\in\CC{L}_2(\BB{R})$ whose Fourier transform is supported inside $[0,\infty)$. To recover the orthogonal complement of $\mathcal{PW}_{[0,\infty)}(\BB{R})$ in $\CC{L}_2(\BB{R})$, namely $\mathcal{PW}_{(-\infty,0]}(\BB{R})$, thereby ensuring that the system is dense in $\CC{L}_2(\BB{R})$, we need to complement it by the orthonormal system generated by the measure $\ee^\xi\D\xi$ for $\xi\in(-\infty,0]$ which, conveniently, we label by $\varphi_n$, $n\leq-1$. The outcome, the MT system, is
\begin{equation}
  \label{eq:3.5}
  \varphi_n(x)=\sqrt{\frac{2}{\pi}} \ii^n \frac{(1+2\ii x)^n}{(1-2\ii x)^{n+1}},\qquad n\in\BB{Z},
\end{equation}
\cite{iserles20for}. The MT system has a number of elegant features:
\begin{Eqnarray*}
  \varphi_n'&=&-n\varphi_{n-1}+\ii(2n+1)\varphi_n+(n+1)\varphi_{n+1},\qquad n\in\BB{Z},\\
  |\varphi_n(x)|&\leq&\sqrt{\frac{2}{\pi}} \frac{1}{(1+4x^2)^{1/2}},\qquad x\in\BB{R},\\
  \varphi_m\varphi_n&=&\frac{1}{\sqrt{2\pi}} (\varphi_{m+n}-\ii \varphi_{m+n+1}),\qquad m,n\in\BB{Z},\\
  2x\varphi_n'&=&-\ii n\varphi_{n-1}-\varphi_n-\ii(n+1)\varphi_{n+1},\\
  \varphi_{n+1}(x)&=& \ii \left(\frac{1+2\ii x}{1-2\ii x}  \right)\varphi_n(x), \\
  \varphi_{-1-n}(x) &=& \ii^{2n-1} \varphi_n(-x).
\end{Eqnarray*}
-- which make its implementation as a spectral basis considerably easier. However, the most valuable feature of the MT system is that, subject to the change of variables $x=\frac12\tan(\theta/2)$, we have
\begin{equation}
  \label{eq:3.6}
  \hat{f}_n=\int_{-\infty}^\infty f(x)\overline{\varphi_n(x)}\D x=\frac{(-\ii)^n}{\sqrt{2\pi}} \int_{-\pi}^\pi \!\left(1-\ii\tan\frac{\theta}{2}\right)\! f\!\left(\frac12\tan\frac{\theta}{2}\right)\! \ee^{-\ii n\theta}\D\theta,\qquad n\in\BB{Z}.
\end{equation}
In other words, {\em the computation of expansion coefficients is equivalent to the evaluation of standard Fourier coefficients of a modified function,\/} a task that can be accomplished (for sufficiently smooth functions) to very high accuracy using the Fast Fourier Transform. 

We note in passing that this feature -- the computation of the first $N$ expansion coefficients in $\O{N\log N}$ operations -- is highly unusual in the setting of Section \ref{subsec:orthosystems}: it can be accomplished only for the MT basis (or its minor generalisation) using FFT and for four other `tanh-Chebyshev' bases using Fast Cosine (or Sine) Transform \cite{iserles20fco}. 

Let us now investigate the FSEs $\Psi(t)$. For simplicity we consider this only for $n\in\BB{Z}_+$, noting that an extension to $n\leq-1$ is straightforward by the symmetry: $\psi_{-1-n}(x,t) = \ii^{2n-1} \psi_{n}(-x,t)$. As before, we assume for the time being that $\varepsilon=1$. Using \R{eq:2.2} we have
\begin{equation}
  \label{eq:3.7}
  \psi_n(x,t)=\frac{(-\ii)^n}{\sqrt{2\pi}} \int_0^\infty \CC{L}_n(\xi)\exp\!\left(-\tfrac{\xi}{2}+\ii t\xi^2+\ii x\xi\right)\!\D\xi,\qquad n\in\BB{Z}_+,
\end{equation}
where $\CC{L}_n$ is the $n$th Laguerre polynomial. We can remove the oscillation from $\exp(\ii x \xi)$ by deforming the contour of integration to the line $\{ z \in \BB{C} : \mathrm{Im}(z) = 2x\mathrm{Re}(z)\}$ (technically by integrating over a the boundary of a sector of radius $R$, where the contribution from the arc decays exponentially in $R$), which yields the formula
\begin{equation}
  \psi_n(x,t) = \sqrt{\frac{2}{\pi}}\frac{(-\ii)^n}{1-2\ii x} \int_0^\infty \CC{L}_n\left( \frac{2s}{1-2\ii x}\right) \exp\left( \frac{4\ii t s^2}{(1-2\ii x)^2} \right) \ee^{-s} \mathrm{d}s.
  \end{equation}
For small values of $t$, this integral is not particularly oscillatory, since $\mathrm{Re}\left( \right)$
It is possible to produce for any specific value of $n$, e.g.
\begin{Eqnarray*}
  \psi_0(x,t)&=&\sqrt{\frac{\ii}{8t}} \exp\!\left(\frac{(2x+\ii)^2}{16\ii t}\right) \CC{erfc}\!\left(\frac{(2x+\ii)}{\sqrt{16 \ii t}}\right), \label{eqn:MTpsi0}\\
  \psi_1(x,t)&=& -\ii \psi_0(x,t) + \left(1 -2 \ii x\right)\frac{\psi_0(x,t) - \psi_0(x,0)}{4t} \label{eqn:MTpsi1}.
\end{Eqnarray*}
There is no need to fear the power of $t$ in the denominator, which cancels as $t\rightarrow0$. We discuss handling this removable singularity in the next subsection.

Fig.~\ref{fig:3.2} displays $|\psi_n|$, $n=0,\ldots,5$, for the MT functions in a setting identical to Fig.~\ref{fig:3.1}. Note that the magnitude for small $ t>0$ varies much more violently for $x>0$ -- obviously, this is reversed for $n\leq-1$ -- and that, like for the Hermite FSE, the magnitude tends to an increasingly regular profile once $t$ grows.

\subsection{A four-term recurrence for the Malmquist--Takenaka FSE basis}

While a closed form expression of the $\psi_n$s is complicated and not clearly even possible, we can derive a useful recurrence formula. Begin from the following differential difference equation for the Laguerre polynomials (which follows by differentiating \cite[18.17.1]{dlmf}),
\begin{equation}
\CC{L}_n(\xi) = \CC{L}_n'(\xi) - \CC{L}_{n+1}'(\xi).
\end{equation}
From this it follows immediately that,
\begin{equation}
\psi_n(x,t) = \frac{(-\ii)^n}{\sqrt{2\pi}} \int_0^\infty \left( \CC{L}_n'(\xi) - \CC{L}_{n+1}'(\xi)\right) \exp\!\left(-\tfrac{\xi}{2}+\ii t\xi^2+\ii x\xi\right)\!\D\xi.
\end{equation}
Integrating by parts, noting that $\CC{L}_n(0) = 1 = \CC{L}_{n+1}(0)$ so the boundary terms vanish,
\begin{equation}
\psi_n(x,t) = \frac{(-\ii)^n}{\sqrt{2\pi}} \int_0^\infty \left( \CC{L}_{n+1}(\xi) - \CC{L}_n(\xi)\right)(2\ii t \xi + \ii x - \tfrac12) \exp\!\left(-\tfrac{\xi}{2}+\ii t\xi^2+\ii x\xi\right)\!\D\xi.
\end{equation}
 We can then use the three-term recurrence,
\begin{equation}
  \label{eq:3.8}
  (n+1)\CC{L}_{n+1}(\xi)=(2n+1-\xi)\CC{L}_n(\xi)-n\CC{L}_{n-1}(\xi),
\end{equation}
to obtain,
\begin{eqnarray*}
\psi_n(x,t) &=&  \frac{(-\ii)^n}{\sqrt{2\pi}} \int_0^\infty \bigg[ 2\ii t\left((2n+3)\CC{L}_{n+1} - (n+1)\CC{L}_n - (n+2)\CC{L}_{n+2}\right) \\
& & \qquad \qquad - 2\ii t\left((2n+1) \CC{L}_n - n\CC{L}_{n-1} - (n+1)\CC{L}_{n+1}\right) \\
& & \qquad\qquad +(\ii x - \tfrac12) (\CC{L}_{n+1}-\CC{L}_n) \bigg] \exp\!\left(-\tfrac{\xi}{2}+\ii t\xi^2+\ii x\xi\right)\!\D\xi\\
&=&  \frac{(-\ii)^n}{\sqrt{2\pi}} \int_0^\infty \bigg[ -2\ii t(n+2)\CC{L}_{n+2} + (2\ii t(3n+4) + \ii x -\tfrac12) \CC{L}_{n+1} \\
& &  - (2\ii t(3n+2) + \ii x - \tfrac12 )\CC{L}_n  + 2\ii t n\CC{L}_{n-1})\bigg] \exp\!\left(-\tfrac{\xi}{2}+\ii t\xi^2+\ii x\xi\right)\!\D\xi \\
&=&  \frac{(-\ii)^n}{\sqrt{2\pi}} \int_0^\infty \bigg[ (-\ii)^2 2\ii t(n+2)\CC{L}_{n+2} - (-\ii) (2 t(3n+4) + x + \tfrac12 \ii ) \CC{L}_{n+1} \\
& &  - (2\ii t(3n+2) + \ii x - \tfrac12 )\CC{L}_n  + (-\ii)^{-1} 2 t n\CC{L}_{n-1})\bigg] \exp\!\left(-\tfrac{\xi}{2}+\ii t\xi^2+\ii x\xi\right)\!\D\xi \\
&=& 2\ii t (n+2) \psi_{n+2}(x,t) - (2t(3n+4) + x+ \tfrac12 \ii) \psi_{n+1}(x,t) \\
& & - (2\ii t(3n+2) + \ii x - \tfrac12)\psi_n(x,t) + 2t n \psi_{n-1}(x,t).
\end{eqnarray*}
Collecting terms yields,
\begin{eqnarray*}
2\ii t (n+2) \psi_{n+2} = (2t(3n+4) + x + \tfrac{1}{2} \ii) \psi_{n+1} + (2\ii t(3n+2) + \ii x + \tfrac12 )\psi_n - 2 t n \psi_{n-1}.
\end{eqnarray*}
We now undo the assignment $\varepsilon = 1$ to obtain the following lemma.
\begin{lemma}
 \label{lem:MTFSErec}
 The FSE corresponding to the MT system obeys the recurrence for $n \geq 1$,
 \begin{eqnarray*}
    \psi_0(x,t) &=&\sqrt{\frac{\ii}{8\varepsilon t}} \exp\!\left(\frac{(2x+\ii)^2}{16\ii \varepsilon t}\right) \CC{erfc}\!\left(\frac{(2x+\ii)}{\sqrt{16 \ii \varepsilon t}}\right),\\
   \psi_1(x,t) &=& -\ii \psi_0 + \left(1 -2 \ii x\right)\frac{\psi_0(x, t) - \psi_0(x,0)}{4\varepsilon t} \\
 \ii(n+1)\psi_{n+1} &=&  \left( 3n+1 + \frac{2x+\ii}{4\varepsilon t} \right)\psi_{n} + \ii\left(3n-1 + \frac{2x - \ii }{4\varepsilon t} \right) \psi_{n-1} - (n-1) \psi_{n-2}.
 \end{eqnarray*}
\end{lemma}


\begin{figure}
  \centering
  \includegraphics[width=\textwidth, trim=30 30 30 30]{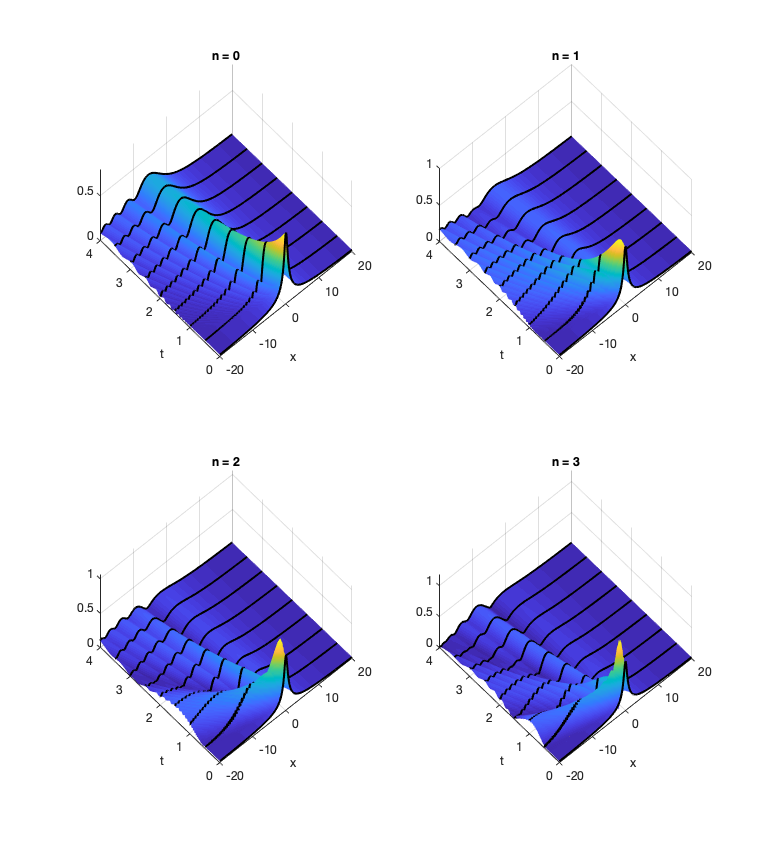}
  \caption{The Malmquist--Takenaka FSE: the functions $|\psi_n(x,t)|$ for $n=0,\ldots,3$, $x \in [-20,20]$ and $t\in[0,4]$}
  \label{fig:3.2}
\end{figure}

Lemma \ref{lem:MTFSErec} indicates the possibility of computing an expansion in the Malmquist--Takenaka FSE basis using the (generalized) Clenshaw algorithm \cite{clenshaw55nsc}. The functions $\psi_n$ for $n\leq-1$ can be addressed using the symmetry $\psi_{-1-n}(x,t) = \ii^{2n-1} \psi_n(-x,t)$, which we omit. Clenshaw's algorithm is best known to apply to bases satisfying three-term recurrences, and in the case of a two-term recurrence reduces to Horner's algorithm. The following lemma spells out the Clenshaw algorithm for a basis with a four-term recurrence (such as the Malmquist--Takenaka FSE).

\begin{lemma}\label{lem:Clenshaw}
  Let $\Phi = \{\varphi_n\}_{n=0}^\infty$ be a basis which satisfies the four-term recurrence,
\begin{equation}\label{eq:3.12} 
\varphi_{n+1}(x) = A_n(x)\varphi_n(x) + B_n(x)\varphi_{n-1}(x) + C_n(x)\varphi_{n-2}(x),
\end{equation}
for $n \geq 1$, where $C_1(x) = 0$, then the finite expansion,
\begin{equation*}
f(x) = \sum_{n=0}^N a_n \varphi_n(x),
\end{equation*}
is equal to 
\begin{equation*}
  v_0(x)\varphi_0(x) + v_1(x)\varphi_{1}(x),
  \end{equation*}
where $\MM{v}(x) = (v_0(x),v_1(x),\ldots, v_{N}(x))^\top$ satisfies the backwards recurrence,
\begin{eqnarray*}
v_N(x) &\!\!\!=\!\!\!& a_N \\
v_{N-1}(x) &\!\!\!=\!\!\!& a_{N-1} + A_{N-1}(x) v_{N}(x) \\
v_{N-2}(x) &\!\!\!=\!\!\!& a_{N-2} + A_{N-2}(x) v_{N-1}(x) + B_{N-1}(x) v_{N}(x) \\
v_{n}(x) &\!\!\!=\!\!\!& a_{n} + A_{n}(x) v_{n+1}(x) + B_{n+1}(x) v_{n+2}(x) + C_{n+2}(x) v_{n+3}(x),
\end{eqnarray*}
for $n = N-3, N-4,\ldots 0$. Where we set $A_0 = 0$.
  \end{lemma}
\begin{proof}
  We follow the derivation of Clenshaw's algorithm in \cite{gautschi2004orthogonal}, but with an extra band below the diagonal in the associated linear system. Indeed, the vector $\boldsymbol{\varphi}(x) = (\varphi_0(x),\ldots,\varphi_N(x))^T$ satisfies
  \begin{equation*}
    \begin{pmatrix}
      1 & & & & & & \\
      -A_0 & 1 & & & & & \\
      -B_1 & -A_1 & 1 & & & & \\
      -C_2 & -B_2 & -A_2 & 1 & & & \\
      0 & -C_3 & -B_3 & -A_3 & 1 & & \\
       & \ddots & \ddots & \ddots & \ddots & \ddots & \\
        & & 0 & -C_{N-1} & -B_{N-1} & -A_{N-1} & 1
      \end{pmatrix} \begin{pmatrix} \varphi_0(x) \\ \varphi_1(x) \\ \varphi_2(x) \\ \varphi_3(x) \\ \varphi_4(x) \\ \vdots \\ \varphi_N(x) \end{pmatrix} = \begin{pmatrix} \varphi_0(x) \\ \varphi_1(x) \\ 0 \\ 0 \\ 0 \\ \vdots \\ 0 \end{pmatrix},
    \end{equation*}
since $A_0 = 0$. Let us write this as $L(x) \boldsymbol{\varphi}(x) = \boldsymbol{\rho}(x)$. Clearly $L(x)$ is invertible, so
\begin{equation*}
  f(x) = \boldsymbol{a}^T \boldsymbol{\varphi}(x) = \boldsymbol{a}^T L(x)^{-1}\boldsymbol{\rho}(x) = \left( L(x)^{-T}\boldsymbol{a} \right)^T \boldsymbol{\rho}(x).
  \end{equation*}
The result is proved by noting that the backward recurrence for $\boldsymbol{v}(x)$ merely computes $\boldsymbol{v}(x) = L(x)^{-T}\boldsymbol{a}$ by back substitution.
\end{proof}

In order to evaluate $\psi_0$ without trouble from the removable singularity, we rewrite equation \eqref{eqn:MTpsi0} in the form
\begin{equation}
\psi_0(x,t) = \varphi_0(x) G_0\left( \frac{2\ii x - 1}{\sqrt{16 \ii \varepsilon t}} \right),
\end{equation}
where $G_0(z) = -\ii \sqrt{\pi} z \ee^{-z^2} \mathrm{erfc}(-\ii z)$. This function is related to $w(z) = \ee^{-z^2}\mathrm{erfc(-\ii z)}$, known as the Faddeeva function or plasma dispersion function \cite{gautschi1970efficient,poppe1990more}. Note that $x,t \in \BB{R}$ corresponds to evaluating $G_0$ in the complex plane in the sector $\{z \in \BB{C} : \mathrm{arg}(z) \in (\pi/4, 5 \pi/4 )  \}$ and we are particularly interested in small positive $t$, which corresponds to large $z$ in this sector. The fact that $G_0(z) \to 1$ as $|z| \to \infty$ within this sector shows the recovery of $\varphi_0(x)$ as $t \to 0$.

Following \cite{gautschi1970efficient,poppe1990more}, the following continued fraction for $G_0$ at $z = \infty$ is convergent in the upper half-plane \cite[7.9.3]{dlmf},
\begin{equation}
G_0(z) = \cfrac{1}{1 - \cfrac{\tfrac12z^{-2}}{1 - \cfrac{z^{-2}}{1 - \cfrac{\tfrac3{2}z^{-2}}{1 - \cfrac{2z^{-2}}{1 - \cdots}}}}}.
\end{equation}
Truncating this continued fraction yields an extremely good approximation for large $z$ in the upper half-plane, and for the lower half-plane we can use the relation \cite[7.4.3]{dlmf}
\begin{equation}
G_0(z) = G_0(-z) -2 \ii \sqrt{\pi} z \ee^{-z^2},
\end{equation}
but note that accuracy can be lost near the complex roots of $\mathrm{erfc}(-\ii z)$ since it relies on heavy cancellation \cite{gautschi1970efficient,poppe1990more}.

In order to evaluate $\psi_1$ without trouble from the removable singularity, we rewrite the formula in Lemma \ref{lem:MTFSErec} in the form
\begin{equation}
\psi_1(x,t) = -\ii \psi_0(x,t) +\sqrt{\frac{2}{\pi}} \frac{2\ii}{(1-2\ii x)^2} G_1\left( \frac{2\ii x - 1}{\sqrt{16 \ii \varepsilon t}} \right),
\end{equation}
where $G_1(z) = 2z^2(G_0(z) - 1)$. While this covers the evaluation of $\psi_0(x,t)$ and $\psi_1(x,t)$ for small $t$, the full implementation of Clenshaw's algorithm may still experience loss of numerical accuracy due to the $1/t$ terms in the recurrence relation. However, numerical issues like this are beyond the scope of this paper.

\section{Bringing the elements together}\label{sec:conclusion}

We bring together the different results of the paper into a cohesive whole. In Section \ref{sec:Zassenhaus}, we reduced the problem of solving the semiclassical Schr\"odinger equation to combining time-steps of the form,
\begin{equation*}
 u^{k+1}(x) =  \ee^{\mathcal{R}_\ell} u^k(x),
  \end{equation*}
where
\begin{Eqnarray*}
  \mathcal{R}_0&=&-\frac12 \tau \varepsilon^{-1} V,\\
  \mathcal{R}_1&=&\frac12 \tau\varepsilon \partial_x^2,\\
  \mathcal{R}_2&=&\black\frac{1}{12}\tau^3\varepsilon \left\{\partial_x^2[V^{(2)}\,\cdot\,]+V^{(2)}\partial_x^2\right\} +\frac{1}{24} \tau^3\varepsilon^{-1} (V^{(1)})^2, \ldots
\end{Eqnarray*}
where $\tau = \ii h$ and $\mathcal{R}_\ell = \O{h^{2\ell - 1}\varepsilon^{-1}}$ for $\ell = 1,2,\ldots$. We propose that the numerical solution be represented implicitly by
\begin{equation*}
  u^k(x) = \sum_{n=0}^N \hat{u}_n \varphi_n(x),
  \end{equation*}
where $\varphi_n$ is either the Hermite function basis or the Malmquist--Takenaka basis (in the latter case the indices should extend from $n = -N-1$ to $n = N$). However, explicitly, we propose that the numerical solution be represented by its values on a grid appropriate to the basis. When this basis is Hermite functions, those points are Hermite quadrature points, and for Malmquist--Takenaka functions, those points are mapped equi-spaced points \cite{weideman1994computation},
\begin{equation}
  x_{j}^{[N]} = \tfrac12\tan\left(\theta_j^{[N]}/2\right), \qquad j = -N-1,\ldots N,
\end{equation}
\begin{equation}
  \theta_j^{[N]} = \frac{j\pi}{N+1}, \qquad j = -N-1,\ldots N.
\end{equation}
We call these \emph{Malmquist--Takenaka points} or \emph{MT points}.

The reason for these choices of grid points are three-fold. First, the mapping from the values of a finite expansion in the basis at these specific grid points, weighted appropriately, to the coefficients in the finite expansion is unitary, so is invertible and perfectly stable. Second, there are known algorithms to compute this mapping and its inverse, which in the case of Malmquist--Takenaka, can be performed rapidly by the Fast Fourier Transform (FFT) and its inverse. Thirdly, at the end of a full time step, we have the solution given by its values on this grid. The computation of the values of the solution at arbitrary points on the real line can be performed stably by barycentric interpolation formula. The barycentric weights for Hermite quadrature points and for equispaced points on the unit circle (which map to MT points) are known explicitly \cite{berrut2004barycentric,wang2014explicit}.

When our solution is represented by values at the grid points, the case $\ell = 0$ is straightforward --- we simply multiply the function value at gridpoint $x_k^{[N]}$ by $\exp(-\tfrac12 \tau \varepsilon^{-1} V(x_k^{[N]}))$.

The case $\ell = 1$ is more subtle, and we propose using the free Schr\"odinger evolutions developed in Section \ref{sec:orthoFSE}. We first compute the coefficients in the $\Phi$ basis, and then evaluate linear combination of those coefficients with the free Schr\"odinger evolution $\Psi(\tfrac12 h \varepsilon)$ at the grid points. This is a two-step process, as follows.
\begin{itemize}
\item Compute the coefficients, $a_0,a_1,\ldots, a_N$ (in the $\Phi$ basis, indexed from $-N-1$ to $N$ in the case of the MT basis) from the values on the grid (using the FFT in the case of the MT basis)
\item  Evaluate the sum $\sum_{k=0}^{N} a_k \psi_k(x,\tfrac12 h \varepsilon)$ at the grid points using Clenshaw's algorithm (in the case of the MT basis, using the 4 term version in Lemma \ref{lem:Clenshaw}).
\end{itemize}

In the case $\ell \geq2$ we propose the use of Krylov subspace methods. This was first proposed in \cite{bader14eas}, later generalised to time-dependent potentials \cite{iserles19sse,iserles18mlm} as well as the method of quasi-Magnus exponential integrators of \cite{blanes17hoc}. There are two facts which make this approach work well. First, $\mathcal{R}_\ell = \O{h^{2\ell-1}\varepsilon^{-1}}$ for $\ell > 1$, so we are computing the exponential of a matrix which is small in spectral norm. As a result, a Krylov subspace with a miniscule dimension can be used \cite{hochbruck1997krylov}. Second, the sparse differentiation matrix (see \eqref{eq:2.1}) implies that the matrices which must be applied to a vector in the Krylov subspace method are a sum of compositions of: diagonal matrices coming from derivatives of the potential function $V$, pentadiagonal matrices coming from the discretisation of $\partial_x^2$ in coefficient space, and transforms between function values on the grid and coefficients (which can be performed using the FFT in the case of the MT basis).

\section*{Acknowledgments}%

This work is partially supported by the Simons Foundation Award No 663281 granted to the Institute of Mathematics of the Polish Academy of Sciences for the years 2021-2023

The authors thank the Isaac Newton Institute for Mathematical Sciences for support and hospitality during the programme ``Geometry, compatibility and structure preservation in computational differential equations'', supported by EPSRC grant  EP/R014604/1, where this work has been initiated.

 Katharina Schratz has received funding from the European Research Council (ERC) under the European Unions Horizon 2020 research and innovation programme (grant agreement No. 850941)
 
The work of Karolina Kropielnicka and of Marcus Webb in this project was financed by The National Center for Science (NCN), based on Grant No. 2019/34/E/ST1/00390

\bibliographystyle{agsm}
\bibliography{Version7}

\end{document}